\documentclass[12pt, reqno]{amsart}
\usepackage{amsmath, amsthm, amscd, amsfonts, amssymb, graphicx, color}
\usepackage[bookmarksnumbered, colorlinks, plainpages]{hyperref}
\hypersetup{colorlinks=true,linkcolor=red, anchorcolor=green, citecolor=cyan, urlcolor=red, filecolor=magenta, pdftoolbar=true}

\newtheorem{theorem}{Theorem}[section]
\newtheorem{lemma}[theorem]{Lemma}
\newtheorem{prop}[theorem]{Proposition}
\newtheorem{cor}[theorem]{Corollary}
\theoremstyle{definition}

\newtheorem{example}[theorem]{Example}

\theoremstyle{remark}
\newtheorem{remark}[theorem]{\bf{Remark}}
\numberwithin{equation}{section}
\allowdisplaybreaks
\begin{document}

\title[ Inequalities for linear functionals and numerical radii on $\mathbf{C}^*$-algebras  ]  
{ Inequalities for linear functionals and numerical radii on $\mathbf{C}^*$-algebras  }



\author[ P. Bhunia ]{ Pintu Bhunia }

\address{(Bhunia) Department of Mathematics, Indian Institute of Science, Bengaluru 560012, Karnataka, India}
\email{  pintubhunia5206@gmail.com / pintubhunia@iisc.ac.in }

\thanks{The author would like to thank SERB, Govt. of India for the financial support in the form of National Post Doctoral Fellowship (N-PDF, File No. PDF/2022/000325) under the mentorship of Prof. Apoorva Khare}


\date{}
\maketitle
\begin{abstract} 
Let $\mathcal{A}$ be  a unital $\mathbf{C}^*$-algebra with unit $e$.
We develop several inequalities for a positive linear functional $f$ on $\mathcal{A}$ and obtain several bounds for the numerical radius $v(a)$ of an element $a\in \mathcal{A}$.
Among other inequalities, we show that if $ a_k, b_k, x_k\in \mathcal{A}$, $r\in \mathbb{N}$ and $f(e)=1$, then
\begin{eqnarray*}
         \left| f \left( \sum_{k=1}^n a_k^*x_kb_k\right)\right|^{r} &\leq& \frac{n^{r-1}}{\sqrt{2}} \left| f\left(    \sum_{k=1}^n \big( (b_k^*|x_k| b_k)^{r}+ i (a_k^*|x_k^*|a_k)^{r} \big)      \right) \right| \quad (i=\sqrt{-1}),
\end{eqnarray*}
\begin{eqnarray*}
         \left| f\left( \sum_{k=1}^n a_k\right)\right|^{2r} &\leq& \frac{n^{2r-1}}{2} f \left(     \sum_{k=1}^n Re(|a_k|^r|a_k^*|^r) + \frac12  \sum_{k=1}^n (|a_k|^{2r}+ |a_k^*|^{2r} )       \right).
    \end{eqnarray*}
We find several equivalent conditions for $v(a)=\frac{\|a\|}{2}$ and $v^2(a)={\frac{1}{4}\|a^*a+aa^*\|}$.
We prove that $v^2(a)={\frac{1}{4}\|a^*a+aa^*\|}$ (resp., $v(a)=\frac{\|a\|}{2}$) if and only if  $\mathbb{S}_{\frac12{ \| a^*a+aa^*\|}^{1/2}} \subseteq V(a) \subseteq \mathbb{D}_{\frac12 {\| a^*a+aa^*\|}^{1/2}}$ (resp., $\mathbb{S}_{\frac12 \| a\|} \subseteq V(a) \subseteq \mathbb{D}_{\frac12 \| a\|}$),
where $V(a)$ is the numerical range of $a$ and $\mathbb{D}_k$ (resp., $\mathbb{S}_k$) denotes the  circular disk (resp., semi-circular disk) with center at the origin and radius $k$.  We also study inequalities for the $(\alpha,\beta)$-normal elements in $\mathcal{A}.$

\end{abstract}

\section{Introduction}
\noindent 
The purpose of this note is to find inequalities for positive linear functionals and study the numerical range and numerical radius on $\mathbf{C}^*$-algebras, which extend some results of classical numerical range and numerical radius, studied in \cite{Bhunia_RMJ, Bhunia_LAA_2021,Bhunia_RIM_2021,Bhunia_BSM_2021,Dragomir-2008, Hirzallah, Kittaneh2005,Kittaneh2003}.  Let us first introduce the necessary notations and terminologies.

\noindent 

Throughout this note, let $\mathcal{A}$  be a unital $\mathbf{C}^*$-algebra with unit $e.$ Let $S(\mathcal{A})$ denote the set of all normalized states, i.e., $ S(\mathcal{A})=\{ f \in \mathcal{A}' : f(e)=\|f\|=1\},$ where $\mathcal{A}'$ denotes the space of all continuous linear functionals on $\mathcal{A}.$ An element $f\in \mathcal{A}'$ is said to be positive if $f(a^*a)\geq 0$ for all $a\in \mathcal{A}$, and then we write $f\geq 0.$ It is well known that $ S(\mathcal{A})=\{ f \in \mathcal{A}' : f\geq 0,\, f(e)=1\}.$
The algebraic numerical range of $a\in \mathcal{A}$, denoted by $V(a)$, is defined as $V(a)=\{ f(a): f\in S(\mathcal{A})\},$ see \cite{Bonsal} and \cite[p. 59]{book1}. It is well known that $V(a)$ is a compact and convex subset of the complex plane $\mathbb{C}$. The algebraic numerical radius of $a\in \mathcal{A}$, denoted by $v(a)$, is defined as $v(a)=\sup\{ |f(a)|: f\in S(\mathcal{A})\}.$ The algebraic numerical radius $v(\cdot): \mathcal{A}\to \mathbb{R}$ defines an equivalent norm on $\mathcal{A}$ via the relation $\frac{1}{2}\|a\| \leq v(a)\leq \|a\|$ for every $a\in \mathcal{A}$. Here the inequalities are sharp, $\frac{1}{2}\|a\| = v(a)$ when $a^2=0$ and $ v(a)= \|a\|$ when $a$ is normal. Like as the usual norm on $\mathcal{A},$ the algebraic numerical radius also satisfies the power inequality, i.e., $v(a^n)\leq v^n(a)$ for all $n=1,2, \ldots.$
For $a\in \mathcal{A}$, let $Re(a)=\frac12 (a+a^*)$ and $Im(a)=\frac1{2i} (a-a^*),$ where $i=\sqrt{-1}$ denotes the imaginary unit. Then clearly $Re(a)$ and $Im(a)$ are the self-adjoint elements in $\mathcal{A}$ and it is easy to verify that $\max_{|\lambda|=1 }\|Re(\lambda a)\|=v(a),$ see in \cite[Lemma 4.2]{LAMA2016}.
Let $|a|$ denote the absolute value of $a\in \mathcal{A}$, i.e.,  $|a|=(a^*a)^{1/2}.$ Then $|a|$ is a self-adjoint (positive) element in $\mathcal{A}$ and $\| |a|\|=\|a\|.$
For a comprehensive account of the algebraic numerical radius in $\mathbf{C}^*$-algebras, the reader can see \cite{Alahmari,Bonsal,Mt3, LAMA2016, Mabrouk}.

\noindent 

If $\mathcal{A}=\mathcal{B}(\mathcal{H})$ is the set of all bounded linear operators on a complex Hilbert space $\mathcal{H}$ and $A\in \mathcal{B}(\mathcal{H})$, then $V(A)$ is the closure of the numerical range $W(A)=\{ \langle Ax,x\rangle : x\in \mathcal{H}, \, \|x\|=1\}$, where $\langle\cdot,\cdot\rangle$ and $\|\cdot\|$ denote as the usual inner product and its norm in $\mathcal{H},$ respectively. The numerical range $W(A)$ is a bounded convex (not necessarily closed) subset of $\mathbb{C}$. For finite dimensional Hilbert space $\mathcal{H}$, $W(A)$ is closed.
The numerical radius of $A$ is $w(A)=\sup\{ |\langle Ax,x\rangle | : x\in \mathcal{H}, \, \|x\|=1\}$, which coincides with the algebraic numerical radius $v(A).$ To study the numerical range and the numerical radius of bounded linear operators, we refer to see \cite{book1, book2}.
\noindent The concept of numerical range and the associated numerical radius of a bounded linear operator defined on a complex Hilbert space are very useful in investigating many analytic and geometric properties of the operator and the corresponding space. The numerical range (also the numerical radius) has found its applications in many fields of sciences, such as functional analysis (estimates of norms), operator theory (differential operators), matrix theory (location of eigenvalues), system theory (singular values), numerical analysis (convergence rate of algorithms), quantum information theory (quantum channels), quantum computing (quantum error correction) and quantum control (optimizing witnesses). Because of the importance of the numerical radius many researchers have been trying to study better approximation of it,
we refer to see \cite{Bhunia-AdM, BhuniaLAA-1, Bhunia-IJPA, Bhunia-IJPA2,  Conde-RIM,Nayak,SababhehCAOT,SababhehFAA, Sahoo} and the references therein. 
There are many generalizations of the classical numerical range and numerical radius, and there has been a great interest in their systematic properties and applications, we refer to see \cite{Mt1, Mt2, BJMA2024, Bhunia-Kittaneh, Mabrouk, Mt11, Mt3,Mt4,Mt5,Mt6,Mt7,Mt8,Mt9}.

\noindent

In this note, we develop several inequalities for the algebraic numerical radius by developing the inequalities of positive linear functionals, and obtain equality characterizations. These results extend the results related to the classical numerical range $W(A)$ and the numerical radius $w(A)$ of a bounded linear operator $A$ acting on a complex Hilbert space.
In Section \ref{sec2}, we provide lower bounds for $v(a)$, which 
improve the bounds $\frac{\|a\|}{2}\leq v(a)$ and   $\frac{1}{4}\|a^*a+aa^*\| \leq v^2(a).$ We also study equality characterizations for these inequalities. Among other  characterizations, 
 we show that $v(a)=\sqrt{\frac{1}{4}\|a^*a+aa^*\|}$ (resp., $v(a)=\frac{\|a\|}{2}$) if and only if  $\mathbb{S}_{\frac12{ \| a^*a+aa^*\|}^{1/2}} \subseteq V(a) \subseteq \mathbb{D}_{\frac12 {\| a^*a+aa^*\|}^{1/2}}$ (resp., $\mathbb{S}_{\frac12 \| a\|} \subseteq V(a) \subseteq \mathbb{D}_{\frac12 \| a\|}$),
where $\mathbb{D}_k$ (resp., $\mathbb{S}_k$) denotes the  circular disk (resp., semi-circular disk) with center at the origin and radius $k$. We provide a formula $v(a)=\frac{1}{{\sqrt{2}}} \max_{|\lambda|=1}\left\| {Re(\lambda a)\pm Im(\lambda a)}\right\|.$ Here we also study some inequalities for the $(\alpha,\beta)$-normal elements in $\mathcal{A}.$ For $0\leq \alpha\leq 1\leq \beta,$ an element $a\in \mathcal{A}$ is said to be $(\alpha,\beta)$-normal if $\alpha^2 f(a^*a)\leq f(aa^*)\leq \beta^2f(a^*a)$ for all $f\in \mathcal{A}'$ with $f\geq 0.$
In Section \ref{sec3}, we obtain upper bounds for $v(a)$, which 
improve as well as generalize the bounds
$v(a)\leq \|a\|$ and   $ v^2(a)\leq \frac{1}{2}\|a^*a+aa^*\|$. 
We also study equality conditions of the above inequalities. 
In Section \ref{sec4}, we study the algebraic numerical radius inequalities for the sum of $n$ elements, sum of $n$ products of elements and commutators of elements.

\section{ Lower bounds for the algebraic numerical radius   }\label{sec2}

\noindent We begin this section with the following lower bounds of $v(a)$, using the decomposition $a=Re(a)+i Im(a)$, which improve the well known lower bound $v(a)\geq \frac{\|a\|}{2}$ and extend some inequalities in \cite{Bhunia_RMJ, Bhunia_LAA_2021, Hirzallah}

\begin{theorem}\label{th1-1}
    Let $a\in \mathcal{A}.$ Then the following inequalities hold:
    \begin{eqnarray}\label{T1E1}
      v(a) &\geq&  \frac{ \|a\|}{2}+ \frac{| \|Re(a)\|- \|Im(a)\| |}{2}.
    \end{eqnarray}
    \begin{eqnarray}\label{T1E2}
      v(a) &\geq&  \frac{\|a\|}{2}+  \frac{\left| \frac{\|a\|}{2}-\|Re(a)\|\right|}{4}+  \frac{\left| \frac{\|a\|}{2}-\|Im(a)\|\right|}{4}.
    \end{eqnarray}
     \begin{eqnarray}\label{T1E3}
      v(a) &\geq&  \frac{\|a\|}2 + \frac1{2\sqrt{2}} \left|  \| Re(a)+Im(a) \| -\|Re(a)-Im(a)\| \right| .
    \end{eqnarray}
\end{theorem}

\begin{proof}
Let $f\in S(\mathcal{A}).$ Then we see that $|f(a)|^2= |f(Re(a))|^2+|f(Im(a))|^2.$
This gives $|f(a)|\geq  |f(Re(a))|$ and $|f(a)|\geq  |f(Im(a))|.$ Therefore, $v(a)\geq \|Re(a)\|$ and $v(a)\geq \|Im(a)\|.$ So,
\begin{eqnarray*}
    v(a)&\geq&  \frac{ \|Re(a)\|+ \|Im(a)\|}{2}+ \frac{| \|Re(a)\|- \|Im(a)\| |}{2}\\
     &\geq& \frac{ \|Re(a)+ i Im(a)\|}{2}+ \frac{| \|Re(a)\|- \|Im(a)\| |}{2}\\
     &=& \frac{ \|a\|}{2}+ \frac{| \|Re(a)\|- \|Im(a)\| |}{2}.
\end{eqnarray*}
Take $\alpha= \max \left\{\frac{\|a\|}{2}, \|Re(a)\| \right\}$ and $\beta= \max \left\{\frac{\|a\|}{2}, \|Im(a)\| \right\}.$ Clearly, we have
\begin{eqnarray*}
    v(a) &\geq&  \frac{\alpha+\beta}{2}+ \frac{|\alpha-\beta|}{2}\\
    &=& \frac12 \frac{\frac{\|a\|}{2}+\|Re(a)\|}{2}+  \frac12 \frac{\left| \frac{\|a\|}{2}-\|Re(a)\|\right|}{2}\\
   && + \frac12 \frac{\frac{\|a\|}{2}+\|Im(a)\|}{2}+  \frac12 \frac{\left| \frac{\|a\|}{2}-\|Im(a)\|\right|}{2}+\frac{|\alpha-\beta|}{2}\\
    &=& \frac{\|a\|}{4}+ \frac{\|Re(a)\|+\|Im(a)\|}{4}+  \frac{\left| \frac{\|a\|}{2}-\|Re(a)\|\right|}{4}+  \frac{\left| \frac{\|a\|}{2}-\|Im(a)\|\right|}{4}+\frac{|\alpha-\beta|}{2}\\
    &\geq& \frac{\|a\|}{4}+ \frac{\|Re(a)+i Im(a)\|}{4}+  \frac{\left| \frac{\|a\|}{2}-\|Re(a)\|\right|}{4}+  \frac{\left| \frac{\|a\|}{2}-\|Im(a)\|\right|}{4}+\frac{|\alpha-\beta|}{2}\\
     &=& \frac{\|a\|}{2}+  \frac{\left| \frac{\|a\|}{2}-\|Re(a)\|\right|}{4}+  \frac{\left| \frac{\|a\|}{2}-\|Im(a)\|\right|}{4}+\frac{|\alpha-\beta|}{2}\\
      &\geq& \frac{\|a\|}{2}+  \frac{\left| \frac{\|a\|}{2}-\|Re(a)\|\right|}{4}+  \frac{\left| \frac{\|a\|}{2}-\|Im(a)\|\right|}{4}.
\end{eqnarray*}
 For any $f\in S(\mathcal{A}),$ we have 
\begin{eqnarray*}
    |f(a)| &=& \sqrt{|f(Re(a))|^2+|f(Im(a))|^2}\\
    &\geq& \frac{1}{\sqrt{2}} ( |f(Re(a))|+|f(Im(a))|  )\\
     &\geq& \frac{1}{\sqrt{2}} ( |f(Re(a)) \pm f(Im(a))|  )\\
     &=& \frac{1}{\sqrt{2}}  |f( Re(a) \pm Im(a)) |.  
\end{eqnarray*}
Taking the supremum over all $f\in S(\mathcal{A})$, we get
\begin{eqnarray}\label{need}
    v(a) &\geq&  \frac{1}{\sqrt{2}} \|  Re(a) \pm Im(a) \|. 
\end{eqnarray}
Therefore,
\begin{eqnarray*}
   \sqrt{2} v(a) &\geq&   \max \{ \| Re(a) + Im(a) \|, \| Re(a) - Im(a) \|\}\\
   &=& \frac{1}{2} (\| Re(a) + Im(a) \| + \| Re(a) - Im(a) \|) \\
    && +\frac12 |\| Re(a) + Im(a) \|- \| Re(a) - Im(a) \||\\
    &\geq & \frac{1}{2} (\| Re(a) + Im(a) + i ( Re(a) - Im(a) )\|) \\
    && +\frac12 |\| Re(a) + Im(a) \|- \| Re(a) - Im(a) \||\\
    &= & \frac12 \|(1+ i) a^*\|
    +\frac12 |\| Re(a) + Im(a) \|- \| Re(a) - Im(a) \||\\
    &= & \frac1{\sqrt{2}}\| a\|
    +\frac12 |\| Re(a) + Im(a) \|- \| Re(a) - Im(a) \||.
\end{eqnarray*}
This completes the proof.
\end{proof}

 To illustrate the bounds in Theorem \ref{th1-1} we consider the following example. 
\begin{example}\label{example1}
    Suppose $X=[-1,1]$ and $l^{\infty}(X)$ denotes the set of all bounded complex valued functions on $X.$ This is a unital $\mathbf{C}^*$-algebra, where the algebraic operations of addition, scalar multiplication and multiplication are all point-wise. The $*$ is point-wise complex conjugation. The norm is the usual supremum norm, that is, 
$\|a\|=\sup \{ |a(x)|: x\in X \}$ for any $a\in l^{\infty}(X).$ Consider $a(x)=x+2ix$, $x\in [-1,1].$ Here $v(a)=\sqrt{5}$, $\|a\|=\sqrt{5}$, $\|Re(a)\|=1$, $\|Im(a)\|=2$, $\| Re(a)+Im(a)\|=3$ and $\| Re(a)-Im(a)\|=1.$ Therefore, the inequality \eqref{T1E1} gives $v(a)\geq \frac{\sqrt{5}}{2}+\frac12$, \eqref{T1E2} gives $v(a)\geq  \frac{\sqrt{5}}{2}+\frac14$ and \eqref{T1E3} gives $v(a)\geq  \frac{\sqrt{5}}{2}+\frac1{\sqrt{2}}$.  This numerical example shows that the inequalities in Theorem \ref{th1-1} give different lower bounds for the algebraic numerical radius of elements in $\mathcal{A}$.
\end{example}

The inequality \eqref{T1E1} also follows from \cite[Theorem 3.10]{Mabrouk}.
From Theorem \ref{th1-1}, we obtain a complete characterization for the equality $v(a)=\frac{\|a\|}{2}.$

 \begin{cor}\label{cor1-1}
   Let $a\in \mathcal{A}$. Then $v(a)=\frac{\|a\|}{2}$ if and only if $\| Re(\lambda a)\|=\| Im(\lambda a)\|=\frac{\|a\|}{2} \text{ for all } |\lambda|=1.$
\end{cor}

\begin{proof}
    From the inequality \eqref{T1E2}, we get 
     \begin{eqnarray}\label{}
      v(a) &\geq& \frac{\|a\|}{2}+  \frac{\left| \frac{\|a\|}{2}-\|Re(\lambda a)\|\right|}{4}+  \frac{\left| \frac{\|a\|}{2}-\|Im(\lambda a)\|\right|}{4} \geq \frac{\|a\|}{2},
    \end{eqnarray}
    for all $|\lambda|=1.$  Therefore, $v(a)=\frac{\|a\|}{2}$ implies  $\| Re(\lambda a)\|= \| Im(\lambda a)\|=\frac{\|a\|}{2}$ for all $|\lambda|=1.$ The converse is trivial, since $\max_{|\lambda|=1 }\|Re(\lambda a)\|=v(a).$ 
\end{proof}

Again, we obtain another complete characterization for the equality $v(a)=\frac{\|a\|}{2}.$ First we prove the following formula for $v(a)$.


\begin{prop}\label{prop1-1}
  Let $a\in \mathcal{A}$.  Then 
  \begin{eqnarray*}
      v(a)&=&\frac{1}{{\sqrt{2}}} \max_{|\lambda|=1}\left\| {Re(\lambda a)\pm Im(\lambda a)}\right\|.
  \end{eqnarray*}
\end{prop}
\begin{proof}
   Since $Re(\lambda a)\pm Im(\lambda a)=Re((1\pm i)\lambda a)$ and $\max_{|\lambda|=1} \sup_{f\in S(\mathcal{A})} |f(Re(\lambda a))|=v(a)$,  we have
   \begin{eqnarray*}
       \frac{1}{{\sqrt{2}}} \max_{|\lambda|=1}\left\| {Re(\lambda a)\pm Im(\lambda  a)}\right\| &=& \frac1{\sqrt{2}}\max_{|\lambda|=1} v\left( {Re(\lambda a)\pm Im(\lambda a)}\right) \\
       &=& \frac1{\sqrt{2}}\max_{|\lambda|=1} \sup_{f\in S(\mathcal{A})} \left| f\left( {Re(\lambda a)\pm Im(\lambda a)} \right) \right|\\
        &=& \frac1{\sqrt{2}}\max_{|\lambda|=1} \sup_{f\in S(\mathcal{A})} \left| f\left( {Re((1\pm i)\lambda a)}\right) \right|\\
         &=& \frac1{\sqrt{2}}\max_{|\lambda|=1} \sup_{f\in S(\mathcal{A})} \left| f \left( {Re((\lambda (1\pm i) a))} \right) \right|\\
          &=& \frac1{\sqrt{2}} v((1\pm i)a)=v(a).
   \end{eqnarray*}
\end{proof}

\begin{cor}\label{cor1-3}
    Let $a\in \mathcal{A}$. Then $v(a)=\frac{\|a\|}{2}$ if and only if $ \|Re(\lambda a) \pm Im(\lambda a)\|= \frac{\|a\|}{\sqrt{2}}$ for all  $|\lambda|=1.$
\end{cor}

\begin{proof}
From the proof of the inequality \eqref{T1E3}, we get
    \begin{eqnarray*}\label{}
      v(a) &\geq&  \frac{1}{\sqrt2} \max \{ \| Re(\lambda a) \pm  Im(\lambda a) \|\}\\
   &\geq&  
   \frac{\|a\|}2 + \frac1{2\sqrt{2}} \left|  \| Re(\lambda a)+Im(\lambda a) \| -\|Re(\lambda a)-Im(\lambda a)\| \right| \\
   &\geq& \frac{\|a\|}{2}, \quad \text{ for all } |\lambda|=1.
    \end{eqnarray*}
    From the above inequalities it is clear that $v(a)=\frac{\|a\|}{2}$ implies $ \|Re(\lambda a) \pm Im(\lambda a)\|= \frac{\|a\|}{\sqrt{2}}$ for all  $|\lambda|=1.$ The converse follows from Proposition \ref{prop1-1}.
\end{proof}

Again, by using the decomposition $a=Re(a)+i Im(a)$, we develop lower bounds of $v(a)$, which improve the bound $v^2(a)  \geq \frac14 \| a^*a+aa^*\|$ and extend some inequalities in \cite{Bhunia_RMJ, Bhunia_LAA_2021}.

\begin{theorem}\label{th1-2}
      Let $a\in \mathcal{A}.$ Then the following inequalities hold:
      \begin{eqnarray}\label{T2E1}
        v^2(a)  &\geq& \frac14 \| a^*a+aa^*\| +\frac12 \left| \|Re(a)\|^2-\|Im(a)\|^2 \right| .
      \end{eqnarray}

\begin{eqnarray}\label{T2E2}
          v^2(a)  &\geq& \frac14 \| a^*a+aa^*\| +\frac14 \left| \|Re(a) +Im(a)\|^2-\|Re(a)-Im(a)\|^2 \right| .
      \end{eqnarray}

       \begin{eqnarray}\label{T2E3}
         v^2(a)  &\geq&  \frac14 \| a^*a+aa^*\| + \frac14(\alpha+\beta),
      \end{eqnarray}
      where $\alpha=  \left| \|Re(a)\|^2-\frac14 \| a^*a+aa^*\| \right|$ and $\beta=  \left| \|Im(a)\|^2-\frac14 \| a^*a+aa^*\| \right| .$
\end{theorem}  

\begin{proof}
   From the inequality $v(a)\geq \max \{ \|Re(a)\|, \|Im(a)\|\}$, we have
\begin{eqnarray*}
    v^2(a)&\geq& \frac{ \|Re(a)\|^2+ \|Im(a)\|^2}{2}+ \frac{| \|Re(a)\|^2- \|Im(a)\|^2 |}{2}\\
    &\geq&  \frac{ \|Re^2(a)\|+ \|Im^2(a)\|}{2}+ \frac{| \|Re(a)\|^2- \|Im(a)\|^2 |}{2}\\
     &\geq& \frac{ \|Re^2(a)+  Im^2(a)\|}{2}+ \frac{| \|Re(a)\|^2- \|Im(a)\|^2 |}{2}\\
     &=& \frac{ \|a^*a+aa^*\|}{4}+ \frac{| \|Re(a)\|^2- \|Im(a)\|^2 |}{2}.
\end{eqnarray*}
From the inequality \eqref{need}, we get
\begin{eqnarray*}
    2v^2(a) &\geq&   \max \{ \| Re(a) + Im(a) \|^2, \| Re(a) - Im(a) \|^2\}\\
   &=& \frac{1}{2} (\| Re(a) + Im(a) \|^2 + \| Re(a) - Im(a) \|^2) \\
    && +\frac12 |\| Re(a) + Im(a) \|^2- \| Re(a) - Im(a) \|^2|\\
    &\geq & \frac{1}{2} (\|(  Re(a) + Im(a) )^2 +  ( Re(a) - Im(a) )^2\|) \\
    && +\frac12 |\| Re(a) + Im(a) \|^2- \| Re(a) - Im(a) \|^2|\\
    &= & \frac12 \|a^*a+aa^*\|
    +\frac12 |\| Re(a) + Im(a) \|^2- \| Re(a) - Im(a) \|^2 |.
\end{eqnarray*}
This implies  the inequality \eqref{T2E2}.  
Take $\gamma=  \max\left\{ \|Re(a)\|^2, \frac14 \| a^*a+aa^*\| \right\}$ and $\delta=  \max \left\{ \|Im(a)\|^2, \frac14 \| a^*a+aa^*\| \right\} .$ Clearly, we have
\begin{eqnarray*}
    v^2(a) &\geq& \max\{ \gamma,\delta\} =\frac{\gamma+\delta}{2}+\frac{|\gamma-\delta|}{2} \\
    &=& \frac{\| Re(a)\|^2+ \|Im(a)\|^2}{4}+ \frac18\|a^*a+aa^*\|+ \frac{1}{4}(\alpha+\beta)+\frac{|\gamma-\delta|}{2}\\
    &\geq & \frac{\| Re^2(a)+ Im^2(a)\|}{4}+ \frac18\|a^*a+aa^*\|+ \frac{1}{4}(\alpha+\beta)+\frac{|\gamma-\delta|}{2}\\
    &=&  \frac14\|a^*a+aa^*\|+ \frac{1}{4}(\alpha+\beta)+\frac{|\gamma-\delta|}{2}\\
     &\geq&  \frac14\|a^*a+aa^*\|+ \frac{1}{4}(\alpha+\beta).
\end{eqnarray*}
This completes the proof.
\end{proof}

 To illustrate the bounds in Theorem \ref{th2-1} we consider the following example. 
\begin{example}\label{example1-}
    Suppose $X=[-1,1]$ and $l^{\infty}(X)$ denotes the set of all bounded complex valued functions on $X.$ This is a unital $\mathbf{C}^*$-algebra. Consider $a(x)=x+2ix$, $x\in [-1,1].$ 
    Then the inequality \eqref{T2E1} gives $v^2(a)\geq \frac{5}{2}+\frac32,$  \eqref{T2E2} gives $v^2(a)\geq \frac{5}{2}+2$ and \eqref{T2E3} gives $v^2(a)\geq \frac{5}{2}+\frac34.$
    This numerical example shows that the inequalities in Theorem \ref{th1-2} give different lower bounds for the algebraic numerical radius of elements in $\mathcal{A}$.
\end{example}

The inequality \eqref{T2E1} also follows from \cite[Theorem 2.10]{Mabrouk}.
From Theorem \ref{th1-2}, we deduce a complete characterization for the equality $v^2(a)=\frac14 \| a^*a+aa^*\|$.
    
\begin{cor}\label{cor1-2}
   Let $a\in \mathcal{A}$. Then $v^2(a)=\frac14 \| a^*a+aa^*\|$ if and only if $\| Re(\lambda a)\|^2=\| Im(\lambda a)\|^2=\frac14 \| a^*a+aa^*\|$ for all $|\lambda|=1.$
\end{cor}

\begin{proof}
From the inequality \eqref{T2E3}, we have 
 \begin{eqnarray}\label{T2E3-0}
        v^2(a) \,\, \geq \,\, \frac14 \| a^*a+aa^*\| + \frac14(\alpha+\beta) &\geq& \frac14 \| a^*a+aa^*\|,
      \end{eqnarray}
      where $\alpha=  \left| \|Re(\lambda a)\|^2-\frac14 \| a^*a+aa^*\| \right|$ and $\beta=  \left| \|Im(\lambda a)\|^2-\frac14 \| a^*a+aa^*\| \right| $, for all $|\lambda|=1.$ Thus,
      $v^2(a)=\frac14 \| a^*a+aa^*\|$ implies  $ \|Re(\lambda a)\|^2= \|Im(\lambda a)\|^2=\frac14 \| a^*a+aa^*\| .$ The converse is trivial, since $\max_{|\lambda|=1 }\|Re(\lambda a)\|=v(a).$ 
\end{proof}

Another complete characterization for the equality $v^2(a)=\frac14 \| a^*a+aa^*\|$ is

\begin{cor}\label{cor1-4}
    Let $a\in \mathcal{A}$. Then $v^2(a)=\frac14 \| a^*a+aa^*\|$ if and only if $ \|Re(\lambda a) \pm Im(\lambda a)\|^2= \frac12 \| a^*a+aa^*\|$ for all  $|\lambda|=1.$
\end{cor}

\begin{proof}
From the proof of the inequality \eqref{T2E2}, we get
    \begin{eqnarray*}\label{}
      v^2(a) &\geq&  \frac{1}{2} \max \{ \| Re(\lambda a) \pm  Im(\lambda a) \|^2\}\\
   &\geq&  
   \frac14 \| a^*a+aa^*\| + \frac1{4} \left|  \| Re(\lambda a)+Im(\lambda a) \|^2 -\|Re(\lambda a)-Im(\lambda a)\|^2 \right| \\
   &\geq&  \frac14 \| a^*a+aa^*\|, \quad \text{ for all } |\lambda|=1.
    \end{eqnarray*}
    Therefore, $v^2(a)= \frac14 \| a^*a+aa^*\|$ implies $ \|Re(\lambda a) \pm Im(\lambda a)\|^2=  \frac12 \| a^*a+aa^*\|$ for all  $|\lambda|=1.$ The converse follows from Proposition \ref{prop1-1}.
\end{proof}

Now, we study complete characterizations for the equalities $v(a)=\frac{\|a\|}{2}$ and $v(a)=\sqrt{\frac14 \| a^*a+aa^*\|}$ in terms of the geometric shape of the algebraic numerical range $V(a)$. 

\begin{lemma}\label{lemma1-1}
    Let $a\in \mathcal{A}$. Let $\mathbb{D}_k$ (resp., $\mathbb{S}_k$) denote the  circular disk (resp., semi-circular disk)  with center at the origin and radius $k$.
    Then   $\| Re(\lambda a) \|=k$ for all $|\lambda|=1$ if and only if 
 $\mathbb{S}_k \subseteq V(a) \subseteq \mathbb{D}_k$.
\end{lemma}

\begin{proof}
    Let  $\| Re(\lambda a) \|=k$ for all $|\lambda|=1$. Then $v( Re(e^{i\theta} a))=\sup_{f\in S(\mathcal{A})}|Re(e^{i\theta}f(a))|=k$ for all $\theta \in \mathbb{R}.$ If possible suppose $V(a)$ does not contain $\mathbb{S}_k.$ Then clearly, there exists $\theta_0\in \mathbb{R}$ such that $v( Re(e^{i\theta_0} a))=\sup_{f\in S(\mathcal{A})}|Re(e^{i\theta_0}f(a))|<k.$ This contradicts the fact $\| Re(\lambda a) \|=k$ for all $|\lambda|=1$. Therefore, $\mathbb{S}_k \subseteq V(a)$. Also, $V(a) \subseteq \mathbb{D}_k$, since $v(a)=k.$ The converse part is clear from  $\mathbb{S}_k \subseteq V(a)$.
\end{proof}


     


 From Lemma \ref{lemma1-1} and  Corollaries \ref{cor1-1}, \ref{cor1-3},  \ref{cor1-2} and  \ref{cor1-4}, we get 


\begin{prop}
    Let $a\in \mathcal{A}$. Then following statements are equivalent:\\
    (i) $v(a)=\frac{\|a\|}{2}.$\\
    (ii) $\| Re(\lambda a) \|= \| Im(\lambda a)\|=\frac{\|a\|}{2}$ for all $|\lambda|=1.$\\
    (iii) $\| Re(\lambda a) \pm Im(\lambda a)\|=\frac{\|a\|}{\sqrt{2}}$ for all $|\lambda|=1.$\\
    (iv) $\mathbb{S}_{\frac12 \| a\|} \subseteq V(a) \subseteq \mathbb{D}_{\frac12 \| a\|},$ where $\mathbb{D}_{\frac12 \| a\|}$ $($resp., $\mathbb{S}_{\frac12 \| a\|})$ denotes the  circular disk $($resp., semi-circular disk$)$  with center at the origin and radius ${\frac12 \| a\|}$.
\end{prop}

\begin{prop}
      Let $a\in \mathcal{A}$. Then following statements are equivalent:\\
    (i) $v(a)=\sqrt{\frac14 \| a^*a+aa^*\|}.$\\
    (ii) $\| Re(\lambda a) \|= \| Im(\lambda a)\|=\sqrt{\frac14 \| a^*a+aa^*\|}$ for all $|\lambda|=1.$\\
    (iii) $\| Re(\lambda a) \pm Im(\lambda a)\|=\sqrt{\frac12 \| a^*a+aa^*\|}$ for all $|\lambda|=1.$\\
    (iv) $\mathbb{S}_{\frac12{ \| a^*a+aa^*\|}^{1/2}} \subseteq V(a) \subseteq \mathbb{D}_{\frac12 {\| a^*a+aa^*\|}^{1/2}},$ where $\mathbb{D}_{\frac12 {\| a^*a+aa^*\|}^{1/2}}$ $($resp., $\mathbb{S}_{\frac12 {\| a^*a+aa^*\|}^{1/2}})$ denotes the  circular disk $($resp., semi-circular disk$)$  with center at the origin and radius ${\frac12 {\| a^*a+aa^*\|}^{1/2}}.$
\end{prop}

For $0\leq \alpha\leq 1\leq \beta,$ an element $a\in \mathcal{A}$ is said to be $(\alpha,\beta)$-normal if $\alpha^2 f(a^*a)\leq f(aa^*)\leq \beta^2f(a^*a)$ for all $f\in \mathcal{A}'$ with $f\geq 0.$ Clearly, every normal element is $(\alpha,\beta)$-normal with $\alpha=\beta=1.$ For the $(\alpha,\beta)$-normal element $a,$ it is easy to verify that $\|a^*a+aa^*\| \geq \max \left\{ 1+\alpha^2, 1+\frac{1}{\beta^2}\right\}\|a\|^2.$ Therefore, from Theorem \ref{th1-2} we get the following inequalities, which extend the results in \cite[Theorem 3.1 and Theorem 3.3]{ASCM}.

\begin{cor}\label{th1-2-1}
      Let $a\in \mathcal{A}$ be $(\alpha,\beta)$-normal. Then the following inequalities hold:
      \begin{eqnarray}\label{T2E1-1}
        v^2(a)  &\geq& \frac14 \max \left\{ 1+\alpha^2, 1+\frac{1}{\beta^2}\right\}\|a\|^2 +\frac12 \left| \|Re(a)\|^2-\|Im(a)\|^2 \right| .
      \end{eqnarray}

\begin{eqnarray*}\label{T2E2-1}
          v^2(a)  &\geq& \frac14 \max \left\{ 1+\alpha^2, 1+\frac{1}{\beta^2}\right\}\|a\|^2 +\frac14 \left| \|Re(a) +Im(a)\|^2-\|Re(a)-Im(a)\|^2 \right| .
      \end{eqnarray*}

       \begin{eqnarray}\label{T2E3-1}
         v^2(a)  &\geq&  \frac14 \max \left\{ 1+\alpha^2, 1+\frac{1}{\beta^2}\right\}\|a\|^2 + \frac14(l+m),
      \end{eqnarray}
      where $l=  \left| \|Re(a)\|^2-\frac14 \| a^*a+aa^*\| \right|$ and $m=  \left| \|Im(a)\|^2-\frac14 \| a^*a+aa^*\| \right| .$
\end{cor}  

Clearly, if $a\in \mathcal{A}$ is $(\alpha,\beta)$-normal, then $v(a)>\frac{\|a\|}{2}$ and the above inequalities gives strictly sharper bounds of $v(a)$ than the bound $v(a)\geq \frac{\|a\|}{2}.$

  \section{ Upper bounds for the algebraic numerical radius }\label{sec3}

\noindent 

We begin this section with stating the well known theorem due to Gelfand, Naimark and Segal \cite[Theorem 1.6.3]{book3}. Recall that a representation $\pi$ of $\mathcal{A}$ is a $*$-homomorphism of $\mathcal{A}$ into $\mathcal{B}(\mathcal{H})$ for some Hilbert space $\mathcal{H}$. Hence, representation $\pi$ satisfies
 $\pi (ab)=\pi(a) \pi (b)$ and $\pi(a^*)={\pi(a)}^*$ for all $a,b \in \mathcal{A},$ where ${\pi(a)}^*$ denotes the Hilbert adjoint of ${\pi(a)}.$

\begin{theorem}\cite[Theorem 1.6.3]{book3}\label{th-lemma}
    Let $f\in \mathcal{A}'$ and $f\geq 0.$ Then there is a representation $\pi$ of $\mathcal{A}$ and a vector $\xi\in \mathcal{H}$ such that $f(a)=\langle \pi (a) \xi, \xi\rangle$ for every $a\in \mathcal{A}.$
 Moreover, if $f\in S(\mathcal{A})$, then $\|\xi\|=1.$
    \end{theorem}

 The following inner product inequalities (mixed Schwarz inequality \cite[pp. 75-76]{Halmos} and McCarthy inequality \cite{McCarthy}) are needed to develop upper bounds for $v(a)$.

\begin{lemma}\label{lem2-1} 
    Let $A\in \mathcal{B}(\mathcal{H}).$ Then the following inequalities hold:
    
   \noindent (i) $|\langle A\xi,\eta \rangle| \leq \sqrt{\langle |A|\xi,\xi \rangle \langle |A^*|\eta,\eta \rangle}$ for every $\xi, \eta\in \mathcal{H}.$

  \noindent  (ii) If $A\geq 0$, then $\langle A\xi,\xi\rangle^p\leq \langle A^p\xi,\xi\rangle$ for every $\xi \in \mathcal{H}$ with $\|\xi\|=1$ and for every $p\geq 1.$
\end{lemma}

 First we obtain the following inequality for positive linear functionals and an upper bound for $v(a)$, which improves $v(a)\leq {\|a\|}$ and extends \cite[Theorem 1]{Kittaneh2003}.
   
 \begin{theorem}\label{th2-1}
      Let $a\in \mathcal{A}$ and  $f\in S(\mathcal{A}).$  Then 
      \begin{eqnarray*}
        |f(a)|^r &\leq &  \frac{1}{2} f(|a|^r+ |a^*|^r) \quad \text{ for all } r=1,2,\ldots.
    \end{eqnarray*}
    
    Moreover, 
      \begin{eqnarray*}
          v^r(a) &\leq& \frac{1}{2} \left\| |a|^r+|a^*|^r \right\|.
      \end{eqnarray*}
      
          In particular, for $r=1$,
            \begin{eqnarray}\label{th2-m1}
          v(a) &\leq& \frac{1}{2} \left\| |a|+|a^*| \right\| \,\, \leq \,\,  \frac{1}{2} \left\| a \right\|+ \frac{1}{2} \sqrt{\left\| a^2 \right\|}.
      \end{eqnarray}
     \end{theorem}
     
\begin{proof}
   From Theorem \ref{th-lemma}, we have
    \begin{eqnarray*}
        |f(a)|^r &=& |\langle \pi(a)\xi,\xi\rangle|^r\\
        &\leq& \langle |\pi(a)|\xi,\xi\rangle^{r/2} \langle |\pi(a^*)|\xi,\xi\rangle^{r/2} \quad (\text{by Lemma \ref{lem2-1} (i)})\\
         &=& \langle \pi(|a|)\xi,\xi\rangle^{r/2} \langle \pi(|a^*|)\xi,\xi\rangle^{r/2} \\
         &\leq & \langle \pi(|a|^r)\xi,\xi\rangle^{1/2} \langle \pi(|a^*|^r)\xi,\xi\rangle^{1/2} \quad (\text{by Lemma \ref{lem2-1} (ii)})\\
         &\leq & \frac12 (\langle \pi(|a|^r)\xi,\xi\rangle +  \langle \pi(|a^*|^r)\xi,\xi\rangle)\\
         &=& \frac12 \langle (\pi(|a|^r)+ \pi(|a^*|^r) )\xi,\xi\rangle\\
         &=& \frac12 \langle \pi(|a|^r+ |a^*|^r) \xi,\xi\rangle\\
         &=& \frac{1}{2} f(|a|^r+ |a^*|^r).
    \end{eqnarray*}
    Therefore, taking the supremum over all $f\in S(\mathcal{A})$, we get
     $v^r(a) \leq \frac{1}{2} \left\| |a|^r+|a^*|^r \right\|.$
    In particular, for $r=1,$ we have
\begin{eqnarray*}
    |f(a)| &\leq& \frac{1}{2} f(|a|+|a^*|) = \frac12 \langle \pi(|a|+ |a^*|) \xi,\xi\rangle\\
    &\leq& \frac12 \| |\pi(a)|+ |\pi(a^*)| \| \\
    &\leq& \frac12 \| \pi(a)\|+ \frac12 \|\pi^2(a) \|^{1/2} \quad (\text{by using \cite[Theorem 1]{Kittaneh2003}})\\
    &=&  \frac12 \| \pi(a)\|+ \frac12 \|\pi(a^2) \|^{1/2} \\
    &\leq& \frac{1}{2} \left\| a \right\|+ \frac{1}{2} \left\| a^2 \right\|^{1/2} \quad (\| \pi(a)\| \leq \|a\|).
\end{eqnarray*}
Hence, taking the supremum over all $f\in S(\mathcal{A})$,
we get the desired bounds.
      \end{proof}



The inequalities in \eqref{th2-m1} are proper. To see this one can consider a $\mathbf{C}^*$-algebra $\mathcal{A}=\mathcal{M}_3(\mathbb{C})$, the algebra of all $3\times 3$ complex matrices and take $a=\begin{bmatrix}
    0&1&0\\
    0&0&2\\ 0&0&0
\end{bmatrix}.$ Then  \begin{eqnarray*}
          v(a)=\frac{\sqrt{5}}{2} &<& \frac32=\frac{1}{2} \left\| |a|+|a^*| \right\| \,\, < \,\, \frac{2+\sqrt{2}}{2}= \frac{1}{2} \left\| a \right\|+ \frac{1}{2} \sqrt{\left\| a^2 \right\|}.
      \end{eqnarray*}
      
\begin{remark}\label{remp-1}
 Let $a\in \mathcal{A}$.
    Then from the inequalities $\frac{\|a\|}{2}\leq v(a) \leq  \frac{1}{2} \left\| a \right\|+ \frac{1}{2} \sqrt{\left\| a^2 \right\|},$ we conclude that if $a^2=0$ then $v(a)=\frac{\|a\|}{2}.$ 
    \end{remark}

Next inequalities give an upper bound for $v(a)$, which improves $v^2(a) \leq  \frac{1}{2}\left\| a^*a+aa^* \right\|$ and $v(a)\leq \|a\|$, and extends \cite[Corollary 2.7]{Bhunia_RIM_2021}.

\begin{theorem}\label{th2-2}
     Let $a\in \mathcal{A}$ and $f\in S(\mathcal{A}).$ Then  
     \begin{eqnarray*}
        |f(a)|^{2r} &\leq& 
         f(t|a|^{2r}+ (1-t) |a^*|^{2r}) \quad \text{ for all $ t\in [0,1]$ and $r\in \mathbb{N}$.}
         \end{eqnarray*}
         
         Moreover, 
      \begin{eqnarray*}
          v^{2r}(a) &\leq&  \left\| t|a|^{2r}+(1-t)|a^*|^{2r} \right\|.
      \end{eqnarray*}
      
       In particular, for $r=1$,
            \begin{eqnarray}\label{p0-0}
          v^2(a) &\leq&  \left\| ta^*a+(1-t)aa^* \right\| . 
      \end{eqnarray}
\end{theorem}

\begin{proof}
     From Theorem \ref{th-lemma}, we have
    \begin{eqnarray*}
        |f(a)|^{2r} &=& |\langle \pi(a)\xi,\xi\rangle|^{2r}\\
        &=& t |\langle \pi(a)\xi,\xi\rangle|^{2r} +(1-t) |\langle \pi(a)\xi,\xi\rangle|^{2r}\\
        &\leq& t\|\pi(a)\xi\|^{2r} + (1-t)\|\pi(a^*)\xi\|^{2r} \quad (\text{by Cauchy-Schwarz inequality})\\
        &=& t \langle \pi(a^*a)\xi,\xi\rangle^{r} +(1-t) \langle \pi(aa^*)\xi,\xi\rangle^{r}\\
        &\leq & t \langle \pi((a^*a)^r)\xi,\xi\rangle^{} +(1-t) \langle \pi((aa^*)^r)\xi,\xi\rangle^{} \quad (\text{by Lemma \ref{lem2-1} (ii)})\\
&= & t \langle \pi(|a|^{2r})\xi,\xi\rangle^{} +(1-t) \langle \pi(|a^*|^{2r})\xi,\xi\rangle^{}\\
&=& t f(|a|^{2r}) +(1-t) f(|a^*|^{2r})\\
&=& f(t|a|^{2r}+ (1-t) |a^*|^{2r}).
         \end{eqnarray*}
        Therefore, taking the supremum over all $f\in S(\mathcal{A})$, we get  
        \begin{eqnarray*}
          v^{2r}(a) &\leq&  \left\| t|a|^{2r}+(1-t)|a^*|^{2r} \right\|.
      \end{eqnarray*}
\end{proof}

From the inequality \eqref{p0-0}, it is clear that
\begin{eqnarray}\label{th2-m1-1}
          v^2(a) &\leq&  \min_{t\in [0,1]}\left\| ta^*a+(1-t)aa^* \right\| \,\, \leq \,\, \frac12 \left\| a^*a+aa^* \right\|. 
      \end{eqnarray}

      The above inequalities in \eqref{th2-m1-1} are proper. To see this one can consider a $\mathbf{C}^*$-algebra $\mathcal{A}=\mathcal{M}_3(\mathbb{C})$, the algebra of all $3\times 3$ complex matrices and take $a=\begin{bmatrix}
    0&1&0\\
    0&0&2\\ 0&0&0
\end{bmatrix}.$ 
      
The following inner product inequality (the Buzano's inequality \cite{Buzano}) is needed to develop next upper bound for $v(a)$.

\begin{lemma}\label{lem2-2} 
    Let $x,y,z\in \mathcal{H}$ with $\|z\|=1$.
    Then $ 2 |\langle x,z\rangle \langle z,y\rangle | \leq {\|x\|\|y\|+ |\langle x,y\rangle|}.$
    \end{lemma}

Using the Buzano's inequality we now obtain the following inequality for positive linear functionals and upper bounds for $v(a)$, which extend \cite[Theorem 2.11]{Bhunia_RIM_2021}.

\begin{theorem}\label{th2-3}
     Let $a\in \mathcal{A}$ and $f\in S(\mathcal{A}).$ Then 
     \begin{eqnarray*}
        |f(a)|^2
        &\leq &  \frac{t}2 |f(a^2)|+  f\left( (1-\frac{3t}{4})a^*a+ \frac{t}4 aa^*\right) \quad  \text{ for all } t\in [0,1].
    \end{eqnarray*}
    
   Furthermore,
     \begin{eqnarray*}
          v^{2}(a) &\leq&  \frac{t}{2}v(a^2)+ \left\| \frac{t}{4}|a^*|^{2}+(1-\frac{3t}{4})|a|^{2} \right\| 
      \end{eqnarray*}
      
      and
      \begin{eqnarray*}
          v^{2}(a) &\leq&  \frac{t}{2}v(a^2)+ \left\| \frac{t}{4}|a|^{2}+(1-\frac{3t}{4})|a^*|^{2} \right\| .
      \end{eqnarray*}
      
 In particular, for $t=1$,
            \begin{eqnarray}\label{p-1}
          v^2(a) &\leq& \frac{1}{2}v(a^2)+\frac14 \left\| a^*a+aa^* \right\|.
      \end{eqnarray}
\end{theorem}

\begin{proof}
     From Theorem \ref{th-lemma}, we get
    \begin{eqnarray*}
        |f(a)|^2 &=& |\langle \pi (a)\xi,\xi\rangle|^2= |\langle \pi (a)\xi,\xi\rangle\langle \xi, \pi (a^*) \xi\rangle|\\
        &\leq& \frac{|\langle \pi (a^2)\xi, \xi \rangle|+ \| \pi(a)\xi\| \|\pi(a^*)\xi \|}{2} \quad (\text{by Lemma \ref{lem2-2}})\\
        &\leq& \frac{|\langle \pi (a^2)\xi, \xi \rangle|+ \frac{1}{2}(\| \pi(a)\xi\|^2+ \|\pi(a^*)\xi\|^2 )}{2}\\
         &=& \frac{|\langle \pi (a^2)\xi, \xi \rangle|+ \frac{1}{2} \langle \pi(a^*a+ aa^*)\xi, \xi \rangle }{2}\\
         &=& \frac12 |f(a^2)|+ \frac14 f(a^*a+aa^*).
    \end{eqnarray*}
    Again, from Theorem \ref{th-lemma}, we get
    \begin{eqnarray*}
        |f(a)|^2 &=& |\langle \pi (a)\xi,\xi\rangle|^2
        \leq \| \pi (a)\xi\|^2=\langle \pi (a^*a)\xi,\xi\rangle=f(a^*a).
    \end{eqnarray*}
    Therefore, 
    \begin{eqnarray*}
        |f(a)|^2&=& t|f(a)|^2 +(1-t) |f(a)|^2 \\
        &\leq & \frac{t}2 |f(a^2)|+ \frac{t}4 f(a^*a+aa^*)+ (1-t) f(a^*a)\\
        &=& \frac{t}2 |f(a^2)|+  f\left( (1-\frac{3t}{4})a^*a+ \frac{t}4 aa^*\right).
    \end{eqnarray*}
    Taking the supremum over all $f\in S(\mathcal{A}),$ we get
    \begin{eqnarray*}
          v^{2}(a) &\leq&  \frac{t}{2}v(a^2)+ \left\| \frac{t}{4}|a^*|^{2}+(1-\frac{3t}{4})|a|^{2} \right\| .
      \end{eqnarray*}
      Replacing $a$ by $a^*$, we get
      \begin{eqnarray*}
          v^{2}(a) &\leq&  \frac{t}{2}v(a^2)+ \left\| \frac{t}{4}|a|^{2}+(1-\frac{3t}{4})|a^*|^{2} \right\|.
      \end{eqnarray*}
\end{proof}

\begin{remark}
(i) The inequality \eqref{p-1} refines and extends   $w^2(A)\leq \frac12 w(A^2)+\frac12 \|A\|^2$ for $A\in \mathcal{B}(\mathcal{H}),$ studied in \cite{Dragomir-2008}.\\ 
(ii) From Theorem \ref{th2-3}, Theorem \ref{th1-2} and  Remark \ref{remp-1}, we obtain that if $a^2=0$ then 
$ v(a) = \sqrt{\frac14 \left\| a^*a+aa^* \right\|}=\frac{\|a\|}{2}.$
\end{remark}

Next inequalities give upper bounds for $v(a)$, which improve $v(a)\leq \frac{1}{2}\| |a| +|a^*|\|$ and $v(a) \leq \|a\|$, and extend \cite[Theorem 2.14]{Bhunia_RIM_2021}.

\begin{theorem}\label{th2-3-}
     Let $a\in \mathcal{A}$ and $f\in S(\mathcal{A}).$ Then  
\begin{eqnarray*}
        |f(a)|^2 &\leq& f\left( t\left(  \frac{|a|+ |a^*|}{2} \right)^2  + (1-t) |a|^2 \right) \quad  \text{ for all } t\in [0,1].
    \end{eqnarray*}
    
     Furthermore, 
      \begin{eqnarray*}
          v^{2}(a) &\leq&   \left\| t\left(\frac{|a|+|a^*|}{2} \right)^2 +(1-t) |a|^{2} \right\| 
      \end{eqnarray*}
      
and \begin{eqnarray*}
          v^{2}(a) &\leq&   \left\| t\left(\frac{|a|+|a^*|}{2} \right)^2 +(1-t) |a^*|^{2} \right\| .
      \end{eqnarray*}
      
 In particular, for $t=1$,
            \begin{eqnarray}\label{p---2}
          v^2(a) &\leq& \frac{1}{2}\left\| Re(|a||a^*|)\right\|+\frac14 \left\| a^*a+aa^* \right\| 
         \,\, \leq \,\, \frac{1}{2}v\left( |a||a^*| \right)+\frac14 \left\| a^*a+aa^* \right\|.
      \end{eqnarray}
\end{theorem}

\begin{proof}
 From Theorem \ref{th-lemma}, we have
    \begin{eqnarray*}
        |f(a)|^2 &=& |\langle \pi(a)\xi,\xi\rangle|^2\\
        &\leq& (\langle |\pi(a)|\xi,\xi\rangle^{1/2} \langle |\pi(a^*)|\xi,\xi\rangle^{1/2} )^2 \quad (\text{by Lemma \ref{lem2-1} (i)})\\
         &=& (\langle \pi(|a|)\xi,\xi\rangle^{1/2} \langle \pi(|a^*|)\xi,\xi\rangle^{1/2})^2 \\
&\leq& \left(\frac12 \langle \pi(|a|)\xi,\xi\rangle+ \langle \pi(|a^*|)\xi,\xi\rangle\right)^2 \\
         &=& \langle \pi \left(\frac12(|a|+ |a^*|) \right)\xi,\xi\rangle^2 \\
         &\leq &\langle \pi \left(\frac12(|a|+ |a^*|) \right)^2\xi,\xi\rangle  \quad (\text{by Lemma \ref{lem2-1} (ii)})\\
         &=& f\left( \left(\frac12(|a|+ |a^*|) \right)^2\right).
    \end{eqnarray*}
     Again, from Theorem \ref{th-lemma}, we can see that
    \begin{eqnarray*}
        |f(a)|^2 &=& |\langle \pi (a)\xi,\xi\rangle|^2
        \leq \| \pi (a)\xi\|^2=\langle \pi (a^*a)\xi,\xi\rangle=f(a^*a).
    \end{eqnarray*}
     Therefore,  
    \begin{eqnarray*}
        |f(a)|^2&=& t|f(a)|^2 +(1-t) |f(a)|^2 \\
        &\leq & tf\left( \left(\frac12(|a|+ |a^*|) \right)^2\right) + (1-t) f(a^*a)\\
        &=& f\left( t\left(  \frac{|a|+ |a^*|}{2} \right)^2  + (1-t) |a|^2 \right).
    \end{eqnarray*}
 Taking the supremum over all $f\in S(\mathcal{A})$, we get
     \begin{eqnarray*}
          v^{2}(a) &\leq&   \left\| t\left(\frac{|a|+|a^*|}{2} \right)^2 +(1-t) |a|^{2} \right\|.
      \end{eqnarray*}
        Replacing $a$ by $a^*$, we get
\begin{eqnarray*}
          v^{2}(a) &\leq&   \left\| t\left(\frac{|a|+|a^*|}{2} \right)^2 +(1-t) |a^*|^{2} \right\|.
      \end{eqnarray*}
     Choosing  $t=1$, we get
            \begin{eqnarray*}\label{}
          v^2(a) &\leq& \frac{1}{2}\left\| Re(|a||a^*|)\right\|+\frac14 \left\| a^*a+aa^* \right\| 
          \leq  \frac{1}{2}v\left( |a||a^*| \right)+\frac14 \left\| a^*a+aa^* \right\|.
      \end{eqnarray*}
      This completes the proof.
\end{proof}

Note that the inequalities $ v^2(a) \leq \frac{1}{2}v(a^2)+\frac14 \left\| a^*a+aa^* \right\|$ (Theorem \ref{th2-3}) and $v^2(a) 
          \leq  \frac{1}{2}v\left( |a||a^*| \right)+\frac14 \left\| a^*a+aa^* \right\|$ (Theorem \ref{th2-3-}) are not comparable, in general.

\begin{remark}
(i) From the first inequality in \eqref{p---2} and Theorem \ref{th1-2}, we obtain that if $Re(|a||a^*|)=0$ then $v^2(a)= \frac14 \left\| a^*a+aa^* \right\|.$\\
   (ii) We show that $v^2(a)
          \leq  \frac{1}{2}v\left( |a||a^*| \right)+\frac14 \left\| a^*a+aa^* \right\|$ is sharper than $v^2(a)\leq \frac12 \|a^*a+aa^*\|.$
        From Theorem \ref{th-lemma} and using the Cauchy-Schwarz inequality we can show that if $f\in S(\mathcal{A})$ and $a,b\in \mathcal{A}$, then $|f(ab)|\leq \frac12 f(aa^*+b^*b)$. This shows 
        $$v(ab) \leq \frac{1}{2} \|aa^*+b^*b\|.$$
        Thus, $ v(|a| |a^*|) \leq \frac{1}{2} \|a^*a+aa^*\|$  and so $\frac{1}{2}v\left( |a||a^*| \right)+\frac14 \left\| a^*a+aa^* \right\|\leq \frac12 \left\| a^*a+aa^* \right\|.$
\end{remark}

Next inequality gives an upper bound for $v(a)$, which also refines $v^2(a)\leq \frac{1}{2}\|a^*a+aa^*\|$ and extends \cite[Theorem 2.4]{Bhunia-GMJ}.

\begin{theorem}\label{th2-4}
     Let $a\in \mathcal{A}$ and $f\in S(\mathcal{A}).$ Then  
 \begin{eqnarray*}
        |f(a)|^2
      &\leq& \frac14  |f(|a|+ i |a^*|)|^2+ \frac14 |f(|a||a^*|)|+ \frac18 f(|a|^2 +|a^*|^2).
    \end{eqnarray*}
    
   Moreover,  
   \begin{eqnarray*}
          v^{2}(a) &\leq&  \frac14 v^2(|a|+i|a^*|) +\frac14 v(|a||a^*|)+ \frac18 \left\| |a|^2+|a^*|^2  \right\|.
      \end{eqnarray*}
\end{theorem}
\begin{proof}
    From Theorem \ref{th-lemma}, we get
    \begin{eqnarray*}
        |f(a)|^2 &=& |\langle \pi(a)\xi,\xi\rangle|^2\\
        &\leq& \langle |\pi(a)|\xi,\xi\rangle^{} \langle |\pi(a^*)|\xi,\xi\rangle^{}  \quad (\text{by Lemma \ref{lem2-1} (i)})\\
        &=&\langle \pi(|a|)\xi,\xi\rangle^{} \langle \pi(|a^*|)\xi,\xi\rangle^{}\\
        &\leq& \frac14 (\langle \pi(|a|)\xi,\xi\rangle^{} + \langle \pi(|a^*|)\xi,\xi\rangle^{})^2\\
        &=& \frac14 (\langle \pi(|a|)\xi,\xi\rangle^{2} + \langle \pi(|a^*|)\xi,\xi\rangle^{2} + 2 \langle \pi(|a|)\xi,\xi\rangle^{} \langle \pi(|a^*|)\xi,\xi\rangle^{})\\
        &=& \frac14 ( |\langle \pi(|a|)\xi,\xi\rangle^{} + i \langle \pi(|a^*|)\xi,\xi\rangle^{} |^2+ 2 \langle \pi(|a|)\xi,\xi\rangle^{} \langle \pi(|a^*|)\xi,\xi\rangle^{})\\
        &\leq& \frac14 ( |\langle \pi(|a|)\xi,\xi\rangle^{} + i \langle \pi(|a^*|)\xi,\xi\rangle^{} |^2+ | \langle \pi(|a|)\xi, \pi (|a^*|)\xi\rangle^{}| +  \| \pi(|a|)\xi\| \|\pi(|a^*|)\xi\|)\\
        && \quad (\text{ by Lemma \ref{lem2-2}})\\
        &\leq& \frac14 ( |\langle \pi(|a|+ i |a^*|)\xi,\xi\rangle^{} |^2+ | \langle \pi(|a||a^*|)\xi, \xi\rangle^{}| +  \frac12 \langle \pi(|a|^2 +|a^*|^2)\xi,\xi\rangle)\\
        &=& \frac14  |f(|a|+ i |a^*|)|^2+ \frac14 |f(|a||a^*|)|+ \frac18 f(|a|^2 +|a^*|^2).
    \end{eqnarray*}
    This gives $v^{2}(a) \leq  \frac14 v^2(|a|+i|a^*|) +\frac14 v(|a||a^*|)+ \frac18 \left\| |a|^2+|a^*|^2  \right\|,$ as desired.
\end{proof}

\begin{remark}
    From Theorem \ref{th-lemma} and using the McCarthy inequality in Lemma \ref{lem2-1}, we have $v^2(|a|+i|a^*|)=\underset{f\in S(\mathcal{A})}{\sup}|f(|a|+i|a^*|)|^2 \leq \underset{f\in S(\mathcal{A})}{\sup} f(|a|^2+|a^*|^2)= \|a^*a+aa^*\|.$ Also, $v(|a||a^*|)\leq \frac12 \|a^*a+aa^*\|.$
    Therefore, $ \frac14 v^2(|a|+i|a^*|) +\frac14 v(|a||a^*|)+ \frac18 \left\| |a|^2+|a^*|^2  \right\|\leq \frac12 \|a^*a+AA^*\|.$
\end{remark}

Next inequality refines  $v(a)\leq \sqrt{\frac12 \|a\|  \big\| |a|+|a^*|\big\|}$ (which follows from Theorem \ref{th2-1} and $v(a)\leq \|a\|$) and extends \cite[Theorem 2.8]{ASCM}.

\begin{theorem}\label{th2-5}
Let $a\in \mathcal{A}$. Then 
\begin{eqnarray*}
    v(a) &\leq& \sqrt{\big( \left\| t|a|+(1-t)|a^*| \right\|\big) \|a\|}, \quad \text {for all } t\in [0,1].
\end{eqnarray*}
\end{theorem}
\begin{proof}
    Let $f\in S(\mathcal{A}).$ Then from Theorem \ref{th-lemma}, we obtain
    \begin{eqnarray*}
        |f(a)|^2 &=& |\langle \pi(a)\xi,\xi\rangle|^2
        \leq \langle |\pi(a)|\xi,\xi\rangle^{} \langle |\pi(a^*)|\xi,\xi\rangle^{}  \quad (\text{by Lemma \ref{lem2-1} (i)})\\
        &\leq& \langle \pi(|a|)\xi,\xi\rangle^{} \langle \pi(|a^*|)\xi,\xi\rangle^{}\\
         &=& \langle \pi(|a|)\xi,\xi\rangle^{t} \langle \pi(|a^*|)\xi,\xi\rangle^{1-t}
         \langle \pi(|a|)\xi,\xi\rangle^{1-t} \langle \pi(|a^*|)\xi,\xi\rangle^{t}\\
          &\leq& \langle (t\pi(|a|)\xi,\xi\rangle^{} + (1-t)\langle \pi(|a^*|)\xi,\xi\rangle^{}) f^{1-t}(|a|) f^{t}(|a^*|) \\
          &=& \langle \pi (t|a|+ (1-t)|a^*|)\xi,\xi\rangle^{} f^{1-t}(|a|) f^{t}(|a^*|) \\
          &=& f(t|a|+ (1-t)|a^*|)  f^{1-t}(|a|) f^{t}(|a^*|) \\
          &\leq& \|t|a|+ (1-t)|a^*|\|  \| |a|\|^{1-t} \| |a^*|\|^{t}\\
          &=& ( \|t|a|+ (1-t)|a^*|\| ) \| a\|.
        \end{eqnarray*}
        This implies  $v^2(a) \leq \big( \left\| t|a|+(1-t)|a^*| \right\|\big) \|a\|,$ as desired.
\end{proof}

From Theorem \ref{th-lemma} and  \cite[Theorem 3.1]{Bhunia_Arxiv24}, we can show if $a\in \mathcal{A}$ and $f\in S(\mathcal{A})$ then
\begin{eqnarray*}
    |f(a)|^n&\leq& \frac{1}{2^{n-1}}|f(a^n)|+\sum_{k=1}^{n-1}\frac{1}{2^k} f^{1/2}(|a^k|^2) f^{ \frac{n-k}2}(f(|a^*|^2)), \quad \text{for every $n=1,2,\ldots.$}
    \end{eqnarray*}
    This gives the following reverse power inequality for $v(a)$.
    \begin{prop}\label{pr1}
        Let  $a\in \mathcal{A}$. Then
        \begin{eqnarray*}
            v^n(a) &\leq& \frac{1}{2^{n-1}}v(a^n)+ \sum_{k=1}^{n-1}\frac{1}{2^k} \|a^k\| \|a\|^{n-k} \quad \text{ for every $n=1,2, \ldots$}
        \end{eqnarray*}
 \end{prop}

\begin{remark}
From Proposition \ref{pr1}, we get 
 \begin{eqnarray*}
            v^n(a) &\leq & \frac{1}{2^{n-1}}v(a^n)+ \sum_{k=1}^{n-1}\frac{1}{2^k} \|a\|^n
            \leq \frac{1}{2^{n-1}}v^n(a)+ \sum_{k=1}^{n-1}\frac{1}{2^k} \|a\|^n \leq \|a\|^n
        \end{eqnarray*}
        and 
         \begin{eqnarray*}
            v^n(a) &\leq & \frac{1}{2^{n-1}}v(a^n)+ \sum_{k=1}^{n-1}\frac{1}{2^k} \|a\|^n
            \leq \frac{1}{2^{n-1}}\|a^n\|+ \sum_{k=1}^{n-1}\frac{1}{2^k} \|a\|^n \leq \|a\|^n.
        \end{eqnarray*}
    Therefore, if $v(a)=\|a\|$ then $v(a^n)=v^n(a)=\|a^n\|=\|a\|^n$, for every $n=1,2,\ldots.$
\end{remark}

\section{Algebraic numerical radius bounds for the sums and products of elements}\label{sec4}

We begin this section with the following inequality for the sum of $n$ positive numbers.

\begin{lemma}\label{lem3-1}\cite{Bhor}
    Suppose $x_1,x_2,\ldots,x_n$ are $n$ positive numbers. Then
    $$ \left (\sum_{k=1}^nx_k \right)^r\leq n^{r-1}\left (\sum_{k=1}^nx_k^r \right) \quad \text{ for every $r\geq 1$}.$$
\end{lemma}

We now find the following inequalities for the positive linear functionals and the algebraic numerical radius of the sum of $n$ elements, which extends \cite[Theorem 2.8]{Bhunia_BSM_2021}.

\begin{theorem}\label{th3-1}
    Let $a_1,a_2,\ldots,a_n\in \mathcal{A}$ and $f\in S(\mathcal{A}).$ Then 
\begin{eqnarray*}
        \left| f\left(\sum_{k=1}^n a_k \right) \right|^{2r} 
&\leq& \frac{n^{2r-1}}{2}  f \left(   \sum_{k=1}^n  \frac{|a_k|^{2r}+ |a^*_k|^{2r}}{2}  +   \sum_{k=1}^n   { Re(|a_k|^r|a^*_k|^r)}    \right),
\end{eqnarray*}

    for every $r\in \mathbb{N}.$ Moreover, 
    \begin{eqnarray*}
        v^{2r}\left( \sum_{k=1}^n a_k\right) &\leq& \frac{n^{2r-1}}{2} \left(    \left\| \sum_{k=1}^n Re(|a_k|^r|a_k^*|^r) \right\|+ \frac12 \left\| \sum_{k=1}^n (|a_k|^{2r}+ |a_k^*|^{2r} )\right\|       \right).
    \end{eqnarray*}
    
    In particular, for $n=1,$
     \begin{eqnarray}
        v^{2r}\left( a \right) &\leq&      \frac12 \left\|  Re(|a|^r|a^*|^r) \right\|+ \frac14 \left\|  |a|^{2r}+ |a^*|^{2r} \right\|, \quad \text{ for } a\in \mathcal{A}.
    \end{eqnarray}
\end{theorem}

\begin{proof}
     For any $a\in \mathcal{A}$, from Theorem \ref{th-lemma}, we get
    \begin{eqnarray*}
        |f(a)|^{2r} &=& |\langle \pi(a)\xi,\xi\rangle|^{2r}\\
        &\leq& \left(\langle |\pi(a)|\xi,\xi\rangle^{r/2} \langle |\pi(a^*)|\xi,\xi\rangle^{r/2} \right)^2 \quad (\text{by Lemma \ref{lem2-1} (i)})\\
         &=& (\langle \pi(|a|)\xi,\xi\rangle^{r/2} \langle \pi(|a^*|)\xi,\xi\rangle^{r/2})^2\\
&\leq & (\langle \pi(|a|^r)\xi,\xi\rangle^{1/2} \langle \pi(|a^*|^r)\xi,\xi\rangle^{1/2})^2 \quad (\text{by Lemma \ref{lem2-1} (ii)})\\
&\leq& \langle \pi \left(\frac12(|a|^r+ |a^*|^r) \right)\xi,\xi\rangle^2 \\
         &\leq &\langle \pi \left(\frac12(|a|^r+ |a^*|^r) \right)^2\xi,\xi\rangle  \quad (\text{by Lemma \ref{lem2-1} (ii)})\\
         &=& \langle \pi \left(\frac14(|a|^{2r}+ |a^*|^{2r}) +\frac12 Re(|a|^r|a^*|^r) ) \right)\xi,\xi\rangle\\
         &=& \frac14 f \left(|a|^{2r}+ |a^*|^{2r} \right) + \frac12 f\left(  Re(|a|^r|a^*|^r)  \right).
    \end{eqnarray*}
    Therefore, we have
    \begin{eqnarray*}
        \left| f\left(\sum_{k=1}^n a_k \right) \right|^{2r} &=& \left| \left(\sum_{k=1}^n f(a_k) \right) \right|^{2r}
        \leq \left(\sum_{k=1}^n |f(a_k)| \right) ^{2r}\\
        &\leq& n^{2r-1}\left(\sum_{k=1}^n |f(a_k)|^{2r} \right) \quad (\text{by Lemma \ref{lem3-1}}) \\
        &\leq& n^{2r-1}  \left( \frac14  \sum_{k=1}^n  f \left(|a_k|^{2r}+ |a^*_k|^{2r} \right) + \frac12  \sum_{k=1}^n  f\left(  Re(|a_k|^r|a^*_k|^r)  \right)  \right)\\
        &=& n^{2r-1}  \left(   \sum_{k=1}^n  f \left(\frac{|a_k|^{2r}+ |a^*_k|^{2r}}{4} \right) +   \sum_{k=1}^n  f\left( \frac{ Re(|a_k|^r|a^*_k|^r)}{2}  \right)  \right)\\
        &=& n^{2r-1}  \left(   f \left( \sum_{k=1}^n  \frac{|a_k|^{2r}+ |a^*_k|^{2r}}{4} \right) +  f\left( \sum_{k=1}^n   \frac{ Re(|a_k|^r|a^*_k|^r)}{2}  \right)  \right).
    \end{eqnarray*}
    Taking the supremum over all $f\in S(\mathcal{A})$, we get
    \begin{eqnarray*}
        v^{2r}\left( \sum_{k=1}^n a_k\right) &\leq& \frac{n^{2r-1}}{2} \left(    \left\| \sum_{k=1}^n Re(|a_k|^r|a_k^*|^r) \right\|+ \frac12 \left\| \sum_{k=1}^n (|a_k|^{2r}+ |a_k^*|^{2r} )\right\|       \right).
    \end{eqnarray*}
\end{proof}

Next we obtain the following inequalities for the sum of $n$ products of elements, which extends \cite[Theorem 2.11]{Bhunia_BSM_2021}.

\begin{theorem}\label{th3-2}
    Let $a_k,b_k,x_k\in \mathcal{A}$ for $k=1,2,\ldots, n$ and  let $f\in S(\mathcal{A}).$  Then 
     \begin{eqnarray*}
        \left| f\left(\sum_{k=1}^n a_k^*x_kb_k \right) \right|^{r} &\leq&
        \frac{n^{r-1}}{\sqrt{2}} \left| f\left(\sum_{k=1}^n   ((b_k^*|x_k| b_k)^r + i (a_k^*|x_k^*| a_k)^r)  \right)\right|,
    \end{eqnarray*}
    
   for every $r\in \mathbb{N}.$  Furthermore, 
    \begin{eqnarray*}
        v^{r}\left( \sum_{k=1}^n a_k^*x_kb_k\right) &\leq& \frac{n^{r-1}}{\sqrt{2}} v\left(    \sum_{k=1}^n \big( (b_k^*|x_k| b_k)^{r}+ i (a_k^*|x_k^*|a_k)^{r} \big)      \right).
    \end{eqnarray*}

    In particular, for $r=1,$
    \begin{eqnarray*}
        v^{}\left( \sum_{k=1}^n a_k^*x_kb_k\right) &\leq& \frac{1}{\sqrt{2}} v\left(    \sum_{k=1}^n \big( b_k^*|x_k| b_k+ i a_k^*|x_k^*|a_k \big)      \right).
    \end{eqnarray*}
    
\end{theorem}

\begin{proof}
From Theorem \ref{th-lemma} and using Lemma \ref{lem2-1}, we obtain
 \begin{eqnarray*}
        \left| f\left(\sum_{k=1}^n a_k^*x_kb_k \right) \right|^{r} &=& \left| \left(\sum_{k=1}^n f(a_k^*x_kb_k) \right) \right|^{r}
        \leq \left(\sum_{k=1}^n |f(a_k^*x_kb_k)| \right) ^{r}\\
        &\leq& n^{r-1}\left(\sum_{k=1}^n |f(a_k^*x_kb_k)|^{r} \right) \quad (\text{by Lemma \ref{lem3-1}}) \\
        &=& n^{r-1}\left(\sum_{k=1}^n |\langle \pi(a_k^*x_kb_k)\xi,\xi\rangle|^{r} \right)\\
         &=& n^{r-1}\left(\sum_{k=1}^n |\langle \pi(a_k^*) \pi(x_k) \pi(b_k)\xi,\xi\rangle|^{r} \right)\\
          &=& n^{r-1}\left(\sum_{k=1}^n |\langle  \pi(x_k) \pi(b_k)\xi,\pi(a_k)\xi\rangle|^{r} \right)\\
          &\leq & n^{r-1}\left(\sum_{k=1}^n \langle  |\pi(x_k)| \pi(b_k)\xi,\pi(b_k)\xi \rangle^{r/2}  \langle  |\pi(x_k^*)| \pi(a_k)\xi,\pi(a_k)\xi \rangle^{r/2}   \right)\\
          &= & n^{r-1}\left(\sum_{k=1}^n \langle  \pi(b_k^*)|\pi(x_k)| \pi(b_k)\xi, \xi \rangle^{r/2}  \langle  \pi(a_k^*)|\pi(x_k^*)| \pi(a_k)\xi,\xi \rangle^{r/2}   \right)\\
           &= & n^{r-1}\left(\sum_{k=1}^n \langle  \pi(b_k^*|x_k| b_k)\xi, \xi \rangle^{r/2}  \langle  \pi(a_k^*|x_k^*| a_k)\xi,\xi \rangle^{r/2}   \right)\\
            &\leq& n^{r-1}\left(\sum_{k=1}^n \langle  \pi((b_k^*|x_k| b_k)^r )\xi, \xi \rangle^{1/2}  \langle  \pi((a_k^*|x_k^*| a_k)^r)\xi,\xi \rangle^{1/2}   \right)\\
            &\leq& \frac{n^{r-1}}{2}\left(\sum_{k=1}^n \langle  \pi((b_k^*|x_k| b_k)^r )\xi, \xi \rangle+ \langle  \pi((a_k^*|x_k^*| a_k)^r)\xi,\xi \rangle   \right)\\
            &\leq& \frac{n^{r-1}}{\sqrt{2}}\left|\sum_{k=1}^n \langle  \pi((b_k^*|x_k| b_k)^r )\xi, \xi \rangle+ i \langle  \pi((a_k^*|x_k^*| a_k)^r)\xi,\xi \rangle   \right| \\
            && \quad (\text{since  } x+y\leq \sqrt{2}|x+iy| \text{ for $x,y\geq 0$})\\
             &=& \frac{n^{r-1}}{\sqrt{2}}\left|\sum_{k=1}^n \langle  \pi((b_k^*|x_k| b_k)^r + i (a_k^*|x_k^*| a_k)^r)\xi,\xi \rangle   \right| \\
              &=& \frac{n^{r-1}}{\sqrt{2}}\left|\sum_{k=1}^n   f((b_k^*|x_k| b_k)^r + i (a_k^*|x_k^*| a_k)^r)  \right|.
    \end{eqnarray*}
    Therefore, taking the supremum over all $f\in S(\mathcal{A})$, we get
    \begin{eqnarray*}
        v^{r}\left( \sum_{k=1}^n a_k^*x_kb_k\right) &\leq& \frac{n^{r-1}}{\sqrt{2}} v\left(    \sum_{k=1}^n \big( (b_k^*|x_k| b_k)^{r}+ i (a_k^*|x_k^*|a_k)^{r} \big)      \right),
    \end{eqnarray*}
    as desired.
\end{proof}

Taking $x_k=e$ in Theorem \ref{th3-2}, we get 
\begin{cor}\label{cor3-1}
    Let $a_k,b_k\in \mathcal{A}$ for $k=1,2,\ldots, n.$ Then 
    \begin{eqnarray*}
        v^{r}\left( \sum_{k=1}^n a_kb_k\right) &\leq& \frac{n^{r-1}}{\sqrt{2}} v\left(   \sum_{k=1}^n \big( |b_k|^{2r}+ i |a_k^*|^{2r} \big)      \right) \quad \text{for $r=1,2,\ldots$}.
    \end{eqnarray*}
    
    In particular, for $r=n=1$,
     \begin{eqnarray}
        v\left(ab\right) &\leq& \frac{1}{\sqrt{2}} v\left(     b^*b+ i aa^* \big)      \right) \quad (a,b\in \mathcal{A}).
    \end{eqnarray}
    
\end{cor}


Putting $a_k=b_k=e$ and  $x_k=a_k$ in Theorem \ref{th3-2}, we get 

\begin{cor}\label{cor3-2}
    Let $a_k\in \mathcal{A}$ for $k=1,2,\ldots, n.$ Then 
    \begin{eqnarray*}
        v^{r}\left( \sum_{k=1}^n a_k \right) &\leq& \frac{n^{r-1}}{\sqrt{2}} v\left(    \sum_{k=1}^n \big( |a_k|^{r}+ i |a_k^*|^{r} \big)      \right) \quad \text{for $r=1,2,\ldots$}.
    \end{eqnarray*}
    
 In particular, for $r=n=1$,
 \begin{eqnarray}\label{p000-}
        v\left( a \right) &\leq& \frac{1}{\sqrt{2}} v\left(      |a|^{}+ i |a^*|^{}       \right) \quad (a\in \mathcal{A}).
    \end{eqnarray}
\end{cor}


We observe that the inequality \eqref{p000-} refines the inequality $v^2(a)\leq \frac{1}{2}\|a^*a+aa^*\|.$
Next inequality for the sum of two elements, which extends \cite[Theorem 2.9]{Bhunia_BSM_2021}.

\begin{theorem}\label{th3-3}
    Let $a,b\in \mathcal{A}.$ Then
    \begin{eqnarray*}\label{n-1}
        v^2(a+b)&\leq& \min \left\{ \|a^*a+b^*b\|, \|a^*a+bb^*\|\right\}+ \min\{ v(ba), v(b^*a) \}+\|a\|\|b\|.
    \end{eqnarray*}
    
    If $a,b$ are self-adjoint, then
    \begin{eqnarray}
        v^2(a+b)&\leq& v^2(a+ib)+ v(ba)+\|a\|\|b\|.
    \end{eqnarray}
\end{theorem}
\begin{proof}
     Let $f\in S(\mathcal{A}).$ Then from Theorem \ref{th-lemma} and using Lemma \ref{lem2-2}, we obtain 
     \begin{eqnarray*}
         |f(a+b)|^2&\leq& (|f(a)|+|f(b)|)^2= |f(a)|^2+|f(b)|^2+2|f(a)f(b)|\\
         &=& |\langle \pi(a)\xi,\xi\rangle|^2+ |\langle \pi(b)\xi,\xi\rangle|^2+2|\langle \pi(a)\xi,\xi\rangle \langle \pi(b)\xi,\xi\rangle|\\
         &\leq& \|\pi(a)\xi\|^2+ \|\pi(b)\xi\|^2+2|\langle \pi(a)\xi,\xi\rangle \langle \xi,\pi(b^*)\xi\rangle|\\
         &\leq&  \langle \pi(a^*a+ b^*b)\xi,\xi \rangle +|\langle \pi(b)\pi(a)\xi,\xi\rangle|+  \| \pi(a)\xi\| \|\pi(b^*)\xi\|\\
          &=&  \langle \pi(a^*a+ b^*b)\xi,\xi \rangle +|\langle \pi(ba)\xi,\xi\rangle|+  \langle \pi(a^*a)\xi,\xi \rangle^{1/2}  \langle \pi(bb^*)\xi,\xi \rangle^{1/2}\\
          &=& f(a^*a+b^*b)+ |f(ba)|+ f^{1/2}(a^*a) f^{1/2}(bb^*).
     \end{eqnarray*}
     Taking the supremum over all $f\in S(\mathcal{A}),$ we get 
     \begin{eqnarray*}
        v^2(a+b)&\leq& \|a^*a+b^*b\|+ v(ba)+\|a\|\|b\|.
    \end{eqnarray*}
    Similarly, we can also show that
    \begin{eqnarray*}
        v^2(a+b)&\leq& \|a^*a+b^*b\|+ v(b^*a)+\|a\|\|b\|.
    \end{eqnarray*}
    Therefore, 
    \begin{eqnarray*}
        v^2(a+b)&\leq& \|a^*a+b^*b\|+ \min\{ v(ba), v(b^*a) \}+\|a\|\|b\|.
    \end{eqnarray*}
    Similarly, again we can show that
     \begin{eqnarray*}
        v^2(a+b)&\leq& \|a^*a+bb^*\|+ \min\{ v(ba), v(b^*a) \}+\|a\|\|b\|.
    \end{eqnarray*}
    Thus, we get the first inequality. 
    Now when $a,b$ are self-adjoint, then
    \begin{eqnarray*}
         |f(a+b)|^2&\leq& (|f(a)|+|f(b)|)^2= |f(a)|^2+|f(b)|^2+2|f(a)f(b)|\\
         &=& |f(a)+i f(b)|^2+2|f(a)f(b)|\\
                  &=& |f(a+i b)|^2+2|f(a)f(b)|\\
        &\leq&|f(a+i b)|^2+ |f(ba)|+ f^{1/2}(a^2) f^{1/2}(b^2).
     \end{eqnarray*}
     Taking the supremum over all $f\in S(\mathcal{A}),$ we get 
     \begin{eqnarray*}
        v^2(a+b)&\leq& v^2(a+ib)+ v(ba)+\|a\|\|b\|.
    \end{eqnarray*}
    This completes the proof.
\end{proof}

If $a, b$ are self-adjoint, then $v(a+b)\leq\sqrt{ v^2(a+ib)+ v(ba)+\|a\|\|b\|}\leq v(a)+v(b)$.
We now obtain an inequality for the commutators of elements, which refines as well as extends the inequality in \cite[Theorem 11]{Fong}.

\begin{theorem}\label{th3-4}
    Let $a,b,x,y\in \mathcal{A}.$ Then
    \begin{eqnarray*}\label{}
        v(axb\pm bya)&\leq& \sqrt{2} \min \left\{ \|b\| \sqrt{\|a^*a+aa^*\|}^{},  \|a\| \sqrt{\|b^*b+bb^*\|}^{} \right\} \max\{ \|x\|, \|y\|\}.
    \end{eqnarray*}
    
    In particular, for $x=y=e,$
     \begin{eqnarray}\label{n-2}
        v(ab\pm ba)&\leq& \sqrt{2} \min \left\{ \|b\| \sqrt{\|a^*a+aa^*\|}^{},  \|a\| \sqrt{\|b^*b+bb^*\|}^{} \right\}.
    \end{eqnarray}
    
     In particular, for $x=e$ and $y=0,$
     \begin{eqnarray}\label{n-3}
        v(ab)&\leq& \sqrt{2} \min \left\{\|b\| \sqrt{\|a^*a+aa^*\|}^{},  \|a\| \sqrt{\|b^*b+bb^*\|}^{} \right\}.
    \end{eqnarray}
    
     In particular, for $ab=ba$,
     \begin{eqnarray}\label{n-4}
        v(ab)&\leq& \frac1{\sqrt{2}} \min \left\{ \|b\| \sqrt{\|a^*a+aa^*\|}^{},  \|a\| \sqrt{\|b^*b+bb^*\|}^{} \right\}.
    \end{eqnarray}
\end{theorem}

\begin{proof}
    Suppose $\|x\|\leq 1, \|y\|\leq 1.$ Take $f\in S(\mathcal{A}).$ From Theorem \ref{th-lemma}, we obtain that
    \begin{eqnarray*}
        |f(ax\pm ya)| &\leq& |f(ax)|+|f(ya)|\\
        &=& |\langle \pi (ax)\xi,\xi\rangle| +  |\langle \pi (ya)\xi,\xi\rangle| \\
        &=& |\langle \pi(a)\pi (x)\xi,\xi\rangle| +  |\langle \pi(y) \pi (a)\xi,\xi\rangle| \\
        &\leq& \| \pi (x)\xi \| \| \pi(a^*)\xi\| +  \| \pi (a)\xi \| \| \pi(y^*)\xi\| \,\, \text{(by Cauchy-Schwarz inequality)}\\
         &\leq&  \| \pi(a^*)\xi\| +  \| \pi (a)\xi \|  \\
          &\leq& \sqrt{2} \sqrt{ \| \pi(a^*)\xi\|^2 +  \| \pi (a)\xi \|^2}  \\
          &=& \sqrt{2} \sqrt{\langle  \pi (aa^*+a^*a)\xi,\xi\rangle }\\
          &=& \sqrt{2} \sqrt{f(aa^*+a^*a)}.
    \end{eqnarray*}
    This implies that $ v(ax\pm ya) \leq \sqrt{2} \sqrt{ \| a^*a+aa^*\|},$ when $\|x\|\leq 1,\, \|y\|\leq 1.$ Consider ${\max\{\|x\|,\|y\|\}}\neq 0$  and replacing $x$ and $y$ by $\frac{x}{\max\{\|x\|,\|y\|\}}$ and $\frac{y}{\max\{\|x\|,\|y\|\}}$, respectively, we get 
     $ v(ax\pm ya) \leq \sqrt{2} \sqrt{ \| a^*a+aa^*\|}\max\{\|x\|,\|y\|\}.$ This inequality is also true when $x=y=0.$ Now, replacing $x$ by $xb$ and $y$ by $by$, we obtain
     $$ v(axb\pm bya) \leq \sqrt{2} \|b\| \sqrt{ \| a^*a+aa^*\|}\max\{\|x\|,\|y\|\}.$$
     Interchanging $a$ and $b$, $x$ and $y$, we also obtain  $$ v(axb\pm bya) \leq \sqrt{2} \|a\| \sqrt{ \| b^*b+bb^*\|}\max\{\|x\|,\|y\|\}.$$ 
     This completes the proof.
\end{proof}

Applying Theorem \ref{th3-4} and Theorem \ref{th1-2}, we obtain the following corollaries.

\begin{cor}\label{cor3-3}
     Let $a,b\in \mathcal{A}$ and let $\alpha(a)=\frac12 {|\|Re(a)\|^2-\|Im(a)\|^2|}$. Then
     \begin{eqnarray*}\label{}
        v(ab\pm ba)&\leq& 2\sqrt{2} \min\left\{\|b\| \sqrt{v^2(a)-\alpha(a)}, \|a\| \sqrt{v^2(b)-\alpha(b)} \right\}\\
        &\leq & 2\sqrt{2} \min\{ \|b\| v(a), \|a\| v(b)\}.
    \end{eqnarray*}
\end{cor}
\begin{proof}
    The proof follows from \eqref{n-2} together with \eqref{T2E1}. 
\end{proof}

\begin{cor}\label{cor3-4}
     Let $a,b\in \mathcal{A}$ and let $\alpha(a)=\frac12 {|\|Re(a)\|^2-\|Im(a)\|^2|}$. Then
     \begin{eqnarray*}\label{n-6}
        v(ab) \leq 2\sqrt{2} \min\left\{\|b\| \sqrt{v^2(a)-\alpha(a)}, \|a\| \sqrt{v^2(b)-\alpha(b)} \right\}
        \leq  2\sqrt{2} \min\{ \|b\| v(a), \|a\| v(b) \}.
    \end{eqnarray*}
    
    In particular, when $ab=ba,$
    \begin{eqnarray*}\label{}
        v(ab)&\leq& \sqrt{2} \min\left\{\|b\| \sqrt{v^2(a)-\alpha(a)}, \|a\| \sqrt{v^2(b)-\alpha(b)} \right\}
        \leq  \sqrt{2} \min\{ \|b\| v(a), \|a\| v(b)\}.
    \end{eqnarray*}
\end{cor}
\begin{proof}
     The proof follows from \eqref{n-3} and \eqref{n-4} together with \eqref{T2E1}.
\end{proof}

\begin{example}\label{example1-3}
    Suppose $X=[-1,1]$ and $l^{\infty}(X)$ denotes the set of all bounded complex valued functions on $X.$ This is a unital $\mathbf{C}^*$-algebra. Consider $a(x)=x+2ix$ and $b(x)=x+ix$, $x\in [-1,1].$ 
    Then \begin{eqnarray*}
        v(ab)=\sqrt{10} < \sqrt{14} & =& \sqrt{2} \min\left\{\|b\| \sqrt{v^2(a)-\alpha(a)}, \|a\| \sqrt{v^2(b)-\alpha(b)} \right\}\\
        & < & \sqrt{20}= \sqrt{2} \min\{ \|b\| v(a), \|a\| v(b)\}.
    \end{eqnarray*}
        \end{example}

\begin{remark}
   (i) From Corollary \ref{cor3-3}, we get if $v(ab\pm ba)=2\sqrt{2}\|b\|v(a)$ then $\| Re(a)\|=\|Im(a)\|.$\\
(ii)  From Corollary \ref{cor3-4}, we get if $ab=ba$ and $v(ab)=\sqrt{2}\|b\|v(a)$ then $\| Re(a)\|=\|Im(a)\|.$\\
(iii) Let $a,b\in \mathcal{A}$ and $ab=ba$. Then $v(ab)\leq 2 v(a)v(b).$ 
Clearly, when either $v(a)=\|a\|$ or $v(b)=\|b\|$, then the bound in Corollary \ref{cor3-4} is strictly sharper than $v(ab)\leq 2 v(a)v(b).$
\end{remark}

We observe that using the inequalities \eqref{T2E2} and \eqref{T2E3} in Theorem \ref{th3-4} one can also deduce other inequalities like Corollary \ref{cor3-3} and Corollary \ref{cor3-4}.

\bigskip
\noindent \textbf{Declarations.} Data sharing not applicable to this article as no datasets were generated or analysed during the current study. Author also declares that there is no financial or non-financial interests that are directly or indirectly related to the work submitted for publication.

\bibliographystyle{amsplain}

\end{document}